\documentclass[11pt]{article}
\usepackage{amsfonts, amsmath, amsthm, amssymb, latexsym, amscd, color}
\usepackage{tabularx,ragged2e,booktabs}
\usepackage{rotating}
\usepackage{mathrsfs,calrsfs}
\usepackage{textcomp}
\usepackage{graphicx,tikz}
\newcommand{\comment}[1]{}
\newif\ifpdf
\ifpdf \else \fi \textwidth = 6.5 in \textheight = 9 in
\newtheorem{thm}{Theorem}[section]

\newtheorem{corollary}[thm]{Corollary}
\newtheorem{lemma}[thm]{Lemma}

\newtheorem{theorem}[thm]{Theorem}

\begin{document}

\title{Connectivity of $3$-distance graphs}
\author{{\small S. R. Musawi$^{a}$, S. H. Jafari}$^{b}$\\
\\{\small $^{a}$Faculty of Mathematical Sciences, Shahrood University of Technology,}\\ {\small  P.O. Box 36199-95161, Shahrood, Iran}\\
{\small Email: r\_musawi@shahroodut.ac.ir, r\_musawi@yahoo.com}\\
\\{\small $^{b}$Faculty of Mathematical Sciences, Shahrood University of Technology,}\\ {\small  P.O. Box 36199-95161, Shahrood, Iran}\\
{\small Email: shjafari@shahroodut.ac.ir, shjafari55@gmail.com }\\
}
\date{}
\maketitle

\begin{abstract}
For a simple graph $G$, the $3$-distance graph, $D_3(G)$, is a graph with the vertex set $V(G)$ and two vertices are adjacent if and only if their distance is $3$ in the graph $G$. For a connected graph $G$, we provide some conditions for the connectedness of $D_3(G)$. Also, we characterize all trees and unicyclic graphs with connected $3$-distance graph.

\end{abstract}


\section{Introduction}

In this paper, we only consider the undirected finite simple graphs $G = (V,E)$ where $V=V(G)$ and $E=E(G)$ are the vertex set and edge set of $G$, respectively.
A graph $H=(V(H),E(H))$ is called a subgraph of $G$, denoted by $H\leqslant G$, if $V(H)\subseteq V(G)$ and 
$E(H)\subseteq E(G)$.
A sequence $x_1x_2\cdots x_t$ of distinct vertices of $G$ (except $x_1$ and $x_t$)  is called a path between two vertices $x_1$ and $x_t$ if  each consequent two vertices are adjacent.  If $x_1=x_t$, then the path  $x_1x_2\cdots x_t$ is called a cycle of graph.
The length of the path $x_1x_2\cdots x_t$ is defined as $t-1$.
 The distance between two vertices of $x,y\in V(G)$, $d(x,y)$ (or $d_G(x,y)$), is the minimum length of all paths between them. The diameter of $G$, $diam(G)$, is the maximum distance between vertices of $G$. 
Also, for $x\in V(G)$ and $H\leqslant G$, $d(x,H)=\underset{h\in H}{\min}\, d(x,h)$ is called the distance of the vertex $x$ and the subgraph $H$.

For a positive integer $k$, the $k$-power of $G$, denoted by $G^k$, is the graph with the vertex set $V(G^k)=V(G)$ and two vertices $x$ and $y$ are adjacent if and only if $d(x,y)\leqslant k$.
The $k$-distance graph of a graph $G$, $D_k(G)$, is a graph with the  vertex set $V(G)$ and two vertices are adjacent if and only their distance is $k$ in the graph $G$. 
One can see that for any graph $G$, $E(D_k(G))=E(G^k)-E(G^{k-1})$. 

The idea of  $k$-distance graphs was first studied by Harary et al. \cite{hhk}. They investigated the connectedness of $2$-distance graph of some graphs. In the book by Prisner  \cite[157-159]{p}, the dynamics of the $k$-distance operator was explored. Furthermore, Boland et al. \cite{bhl} extended the $k$-distance operator to a graph invariant, they called distance-n domination number.
To solve one problem posed by Prisner \cite[Open Problem 29]{p}, Zelinka
\cite{z} constructed a graph that is r-periodic under the $2$-distance operator for each
positive even integer r. Prisner’s problem was completely solved by Leonor in her
master’s thesis \cite{l}, wherein she worked with a graph construction different from
that of Zelinka’s.
 Azimi et al. \cite{af2} studied all graphs whose $2$-distance graphs
have maximum degree 2. They also solved the problem of finding all graphs whose
$2$-distance graphs are paths or cycles. Also, the same authors \cite{af1} determined
all finite simple graphs with no square, diamond, or triangles with a common
vertex that are self $2$-distance graphs (that is, graphs G satisfying the equation
$D_2(G)=G$. They further showed the nonexistence of cubic self $2$-distance
graphs.
In \cite{c}, Ching   gave some characterizations of $2$-distance graphs and found all graphs X such that $D_2(X)=kP_2$
or $K_m \cup K_n$, where $k \geqslant 2$ is an integer, $P_2$ is the path of order $2$, and $K_m$ is the complete graph of order $m \geqslant 1$. 
In \cite{k} Khormali gave some conditions for connectivity $D_k(G)$. For $D_2(G)$, he proved that if $G$ has no odd cycle, then $D_2(G)$ is disconnected.
E. Gaar et al. charactrized all graphs with $2$-regular $2$-distance graph.
Recently, in \cite{jm}, Jafari et al. completely characterized all graphs with connected $2$-distance graph.

In this paper, for  a connected graph $G$  and a 3-induced subgraph $H$ of $G$ (Lemma \ref{lem22}), we prove that, if $D_3(H)$ is connected and for any $x\in V(G)\setminus V(H)$ there is $y\in V(H)$ such that $d(x,y)\geqslant3$, then $D_3(G)$ is connected. Also, we charactrize all trees (Theorem \ref{the27}) and unicyclic graphs (Theorems \ref{the28}, \ref{the29}, \ref{the213}, \ref{the215} and \ref{the216}) with connected $3$-distance graph.

\section{Connectivity of $D_3(G)$}

Let $G$ be a graph and $H$ be a subgraph of $G$. We say $H$ is t-induced subgraph of $G$ if  for any $x,y\in V(H)$ with  $d_G(x,y)\leqslant t$ then $d_H(x,y)=d_G(x,y)$. Clearly, $1$-induced subgraph is the known induced subgraph.

\begin{lemma}\label{lem22}
Let $G$ be a connected graph and $H$ be a 3-induced subgraph of $G$. If $D_3(H)$ is connected and for any $x\in V(G)\setminus V(H)$ there is $y\in V(H)$ such that $d(x,y)\geqslant3$, then $D_3(G)$ is connected.
\end{lemma}
\begin{proof}
Let $x\in V(G)\setminus V(H)$ , $d(x,H)=t$,  $y\in V(H)$ and, $x=x_0x_1\cdots x_t=y$ be a path. We consider three cases for $t$.

Case 1 : $t=3k$.
Then $x=x_0x_3\cdots x_{3k}=y$  is a path in $D_3(G)$.

Case 2 : $t=3k+1$.
Since $x=x_0x_3\cdots x_{3k}$  is a path in $D_3(G)$, we can consider $k=0$. By the hypothesis there is $z\in V(H)$ such that $d(x,z)\geqslant3$. By connectivity of $H$, there is a path $x_1=y=y_0y_1\cdots y_s=z$ in $H$. We claim that there is $y_i$ such that $d(x,y_i)=3$. By the contrary, since $d(x,y_1)\leqslant2$ and $d(x,y_2)\leqslant3$, then $d(x,y_2)\leqslant2$. Similarly, $d(x,y_3)\leqslant2$, $d(x,y_4)\leqslant2$, $\cdots$. Consequently $d(x,y_t)\leqslant2$, a contradiction.

Case 3 : $t=3k+2$.
Since $x=x_0x_3\cdots x_{3k}$  is a path in $D_3(G)$, we can consider $k=0$. By the hypothesis there is $z\in V(H)$ such that $d(x,z)\geqslant3$. By connectivity of $H$, there is a path $x_{2}=y=y_0y_1\cdots y_s=z$ in $H$. We claim that there is $y_i$ such that $d(x,y_i)=3$. By the contrary, since $d(x,y_0)=2$ and $d(x,y_1)\leqslant3$, then $d(x,y_1)=2$. Similarly, $d(x,y_2)=2$, $d(x,y_3)=2$, $\cdots$. Consequently $d(x,y_t)=2$, a contradiction.
\end{proof}

\begin{corollary}\label{cor23}
Let $G$ be a connected graph, $H\leqslant G$, $H$ contains two vertices $a$ and $b$ such that $d_H(a,b)=d_G(a,b)=5$. If $H$ is $3$-induced and $D_3(H)$ is connected, then $D_3(G)$ is connected.
\end{corollary}

\begin{lemma}\label{lem24}
Let $H_t$ be the following graph.
\begin{figure}[h]
\begin{center}
\begin{tikzpicture}
\draw (0.0,0)--(4,0);\draw[dashed] (4,0)--(5,0);\draw (5,0)--(9,0);
\foreach \x in{0,1,...,9} \filldraw (\x,0) circle(2pt);
\draw (2,0)--(2,1);\draw (7,0)--(7,1);
 \filldraw (2,1) circle(2pt); \filldraw (7,1) circle(2pt);
\node at (0,-.3) {$a_1$};\node at (1,-.3) {$a_2$};\node at (2,1.3) {$a_3$};
\node at (9,-.3) {$b_1$};\node at (8,-.3) {$b_2$};\node at (7,1.3) {$b_3$};
\node at (2,-.3) {$u_0$};\node at (3,-.3) {$u_1$};\node at (4,-.3) {$u_2$};
\node at (5,-.3) {$u_{t-2}$};\node at (6,-.3) {$u_{t-1}$};\node at (7,-.3) {$u_t$};
\end{tikzpicture}
\end{center}
\caption{ $H_t$}
\end{figure}
If $gcd(t,3)\ne1$, then $D_3(H_t)$ is connected.

\end{lemma}
\begin{proof}
First, we assume that $t=3k+1$ for some integer $k$.  Obviously,  $a_1u_1u_4\cdots u_t$,  $a_2u_2u_5\cdots u_{t-2}b_2$, $u_0u_3\cdots u_{t-1}b_1$,     $a_1a_3u_2$ and     $b_1b_3u_{t-2}$ are paths in $D_3(H_t)$ which shows that  $D_3(H_t)$ is connected.
The case $t=3k+2$  is simillar to first case.

\end{proof}

For positive integers  $l$, $m$ and $n$, the generalized double star $GDS(l,m,n)$ is the graph consisting of the union of two stars $K_{1,m}$ and $K_{1,n}$ together with a path of the length $l$ joining their centers. In the special case, $GDS(1,m,n)$ is a double star graph.

The following theorem will charactrize all trees with connected $3$-distance graph. Here, we need  some special vertices that we call them \textit{inner nodes}.
In a graph $G$, we say the vertex $u\in V(G)$ is an \textit{inner node} if $deg(u)\geqslant3$ and $u$ has at least two neighbours with degree at least 2.

\begin{theorem}\label{the27}
Let $G$ be a tree. $D_3(G)$ is connected if and only if $G$ has two inner nodes with distance $3k+1$ or $3k+2$ for some integer $k$.
\end{theorem}
\begin{proof}$''\Rightarrow''$
By contradiction let us assume that $D_3(G)$ is connected and  the distance of any two inner nodes is a multiple of 3. We can consider  following cases.\\
Case 1 : 
$G$ has no inner node. Then $G$ is a star or generalized double star. Let $a$ be vertex of $G$ with degree $1$ and $b$ be the unique neighbour of $a$. One can see that $a$ and $b$ lie in two distinct connected components of $D_3(G)$ which shows that $D_3(G)$ is disconnected, a contradiction.\\
Case 2 :
$G$ has a unique inner node $a$. Then by the structure of graph, $G-a$ is the union of connected components which are $K_1$, $K_2$, Stars or $GDS(l,1,n)$ for some $n$. Let $H$ be a connected component of $G-a$. Let $L=\langle H\cup\{a\}\rangle$. Then  $L$ is $K_2$, Stars or $GDS(l,1,n)$ for some $l,n$.
It is obviously that $D_3(L)\leqslant D_3(G)$ and for any  $b\in V(L)$, $a$ and $b$ are in a connected component of $D_3(L)$ if and only if $d(a,b)=3k$ for some integer $k$.
Also, if  $b\in V(L)\setminus\{a\}$, $d(a,b)=3k$ and $c\not\in V(L)$ then $d(b,c)\geqslant4$ which showes that $a$ and $b$ are in a connected component of $D_3(G)$ if and only if $d(a,b)=3k$ for some integer $k$.
Therefore $D_3(G)$ is disconnected, a contradiction.\\
Case 3 :
$G$ has atleast two inner nodes. 
Let $a$ be an inner node of $G$. Obviously, the  subgraph $\langle\{b: d(a,b)=3k, \text{for some\;} k\}\rangle$ is a connected component, a contradiction.

\noindent $''\Leftarrow''$
By the assumption, $G$ has an induced subgraph $H$ isomorphic to $H_t$ for some $t=3k+1$ or $t=3k+2$. By Lemma \ref{lem24}, $D_3(H)$ is connected. 
Since $G$ is a tree and $H$ contains a path with length of at least $5$, for any $y\in V(G)\setminus V(H)$ there is $x\in V(H)$ such that $d(x,y)\geqslant3$. Then by Lemma \ref{lem22}, $D_3(G)$ is connected. 
\end{proof}

A connected graph $G$ is called  a unicyclic graph if it is contained exactly one cycle. For charactrizing all unicyclic graph with connected $3$-distance graph we consider $5$ cases respect to the length of the unique cycle of graph. We start with the simplest case, unicyclic graphs with $C_n, n\geqslant7$ and $gcd(n,3)=1$.

\begin{theorem}\label{the28}
Let $G$ be a unicyclic graph with induced cycle $C_n$. If $n\geqslant7$ and $gcd(n,3)=1$, then $D_3(G)$ is connected.
\end{theorem}
\begin{proof}
Let $H=x_0x_1x_2\cdots x_{n-1}x_0$ be the unique cycle of $G$.
Since $gcd(3,n)=1$,  $\overline{0},\overline{3},\cdots , \overline{3(n-1)}$ are distinct,  where $\overline{i}$ is the remainder of $i$ modulo $n$. Then  $x_{\overline{0}}x_{\overline{3}}\cdots x_{\overline{3n-3}}$ is a path with $n$ vertices in $D_3(H)$, which shows $D_3(H)$ is connected. Since $G$ is unicyclic and $n\geqslant6$, for any $a\in V(G)\setminus V(H)$ there is $b\in V(H)$ such that $d(a,b)\geqslant3$. Hence, by Lemma \ref{lem22}, $D_3(G)$ is connected, and the proof is completed. 
\end{proof}

%%%%%%%%%%%%%%%%%%%%%%

Now, we consider the unicyclic graphs with $C_n, n\geqslant6$ and $gcd(n,3)=3$.

\begin{theorem}\label{the29}
Let $G$ be a unicyclic graph with induced cycle $C_n$. If $n\geqslant6$ and $gcd(n,3)=3$, then $D_3(G)$ is connected if and only if $G$ has two inner nodes with distance $3k+1$ or $3k+2$ for some integer $k$.
\end{theorem}

\begin{proof}$''\Rightarrow''$
By contradiction let us assume that $D_3(G)$ is connected and  the distance of any two inner nodes is a multiple of 3. We can consider  following cases.\\
Case 1 : 
$G$ has no inner node. Then $G$ is a cycle graph $x_0x_1x_2\cdots x_{n-1}x_0$ where $n=3k$ for some integer $k$. Then  $\langle x_0,x_3,\cdots, x_{n-3}\rangle$ is a connected component of $D_3(G)$, which shows that $D_3(G)$ is disconnected, a contradiction.\\
Case 2 :
$G$ has atleast an inner node $a$. Obviously, the  subgraph $\langle\{b: d(a,b)=3k, \text{for some\;} k\}\rangle$ is a connected component, a contradiction.
\\
\noindent
$''\Leftarrow''$
Two following cases is considrable.
Let $G$ has two inner nodes $a$ and $b$ with distance $t=3k+1$ or $t=3k+2$.
\\
Case 1 : $n>6$ or both of $a$ and $b$ are not on the unique cycle of $G$. Then  $G$ has an induced subgraph $H$ isomorphic to $H_t$, for some $t\in\{3k+1,3k+2\}$. By Lemma \ref{lem24}, $D_3(H)$ is connected. By assumption for any $y\in V(G)\setminus V(H)$ there is $x\in V(H)$ such that $d(x,y)\geqslant3$. Then by Lemma \ref{lem22}, $D_3(G)$ is connected. 
\\
Case 2 : 
$n=6$ and $a$ and $b$ lie on the unique cycle of $G$. Then $G$ has an induced subgraph $H$ isomorphic to one of the following graphs.

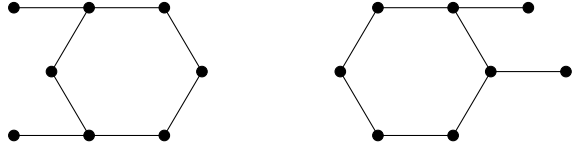
\begin{center}
\begin{tabular}{ccccc}
\begin{tikzpicture}
\draw (0,0)--(0.5,-.85)--(1.5,-.85)--(2,0)--(1.5,.85)--(0.5,.85)--(0,0);
\draw (-.5,-.85)--(.5,-.85);
\draw (-.5,.85)--(.5,.85);
\filldraw (0,0) circle(2pt);\filldraw (0.5,-.85) circle(2pt);\filldraw (1.5,-.85) circle(2pt);\filldraw (2,0) circle(2pt);\filldraw (1.5,.85) circle(2pt);
\filldraw (0.5,.85) circle(2pt);\filldraw (-.5,-.85) circle(2pt);\filldraw (-.5,.85) circle(2pt);
\end{tikzpicture}
&&&&
\begin{tikzpicture}
\draw (0,0)--(0.5,-.85)--(1.5,-.85)--(2,0)--(1.5,.85)--(0.5,.85)--(0,0);
\draw (2,0)--(3,0);
\draw (2.5,.85)--(1.5,.85);
\filldraw (0,0) circle(2pt);\filldraw (0.5,-.85) circle(2pt);\filldraw (1.5,-.85) circle(2pt);\filldraw (2,0) circle(2pt);\filldraw (1.5,.85) circle(2pt);
\filldraw (0.5,.85) circle(2pt);\filldraw (3,0) circle(2pt);\filldraw (2.5,.85) circle(2pt);
\end{tikzpicture}
\end{tabular}
\end{center}

One can see that $D_3(H)$ is connected. Similar to the previous case, $D_3(G)$ is connected. 
\end{proof}

%%%%%%%%%%%%%%%%%%%%%%%%%%%%%%%%%%%%%%%%

Unicyclic graphs containing $C_5$ have more details, so it is necessary to first state some special modes that are expressed in three lemmas.

\begin{lemma}\label{lem210}
Let $G$ be a unicyclic graph. If $G$ has a subgraph isomorphic to one of the following graphs, then $D_3(G)$ is connected.
\end{lemma}

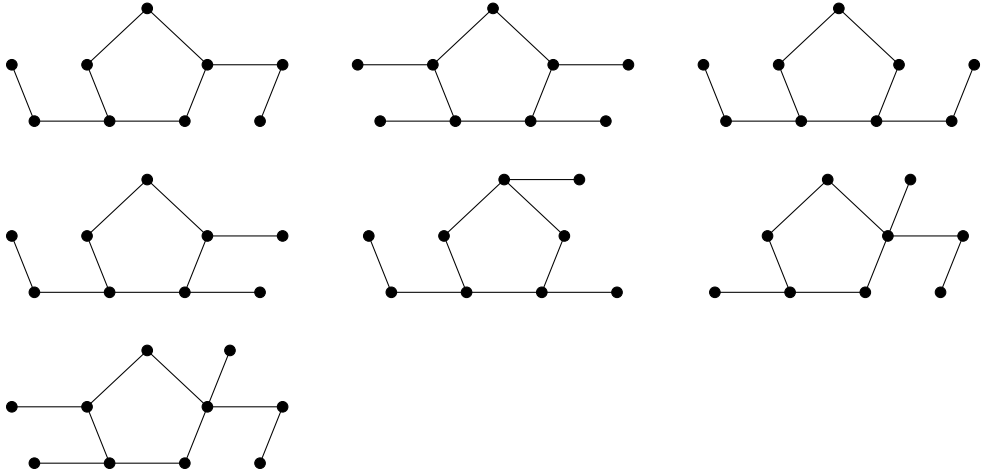
\begin{center}
\begin{tabular}{ccccc}
\begin{tikzpicture}
\draw (0,0)--(1,0)--(1.3,.75)--(.5,1.5)--(-.3,.75)--(0,0);
\draw (-1.3,0.75)--(-1,0)--(0,0);
\draw (1.3,.75)--(2.3,0.75)--(2,0);
\filldraw (0,0) circle(2pt);\filldraw (1.3,.75) circle(2pt);\filldraw (.5,1.5) circle(2pt);\filldraw (1,0) circle(2pt);\filldraw (-.3,.75) circle(2pt);
\filldraw (-1.3,0.75) circle(2pt);\filldraw (-1,0) circle(2pt);\filldraw (2.3,0.75) circle(2pt);\filldraw (2,0) circle(2pt);
\end{tikzpicture}
&&
\begin{tikzpicture}
\draw (0,0)--(1,0)--(1.3,.75)--(.5,1.5)--(-.3,.75)--(0,0);
\draw (-1,0)--(0,0);\draw (2,0)--(1,0);
\draw (1.3,.75)--(2.3,.75);\draw (-.3,.75)--(-1.3,.75);
\filldraw (0,0) circle(2pt);\filldraw (1.3,.75) circle(2pt);\filldraw (.5,1.5) circle(2pt);\filldraw (1,0) circle(2pt);\filldraw (-.3,.75) circle(2pt);
\filldraw (-1.3,0.75) circle(2pt);\filldraw (-1,0) circle(2pt);\filldraw (2.3,0.75) circle(2pt);\filldraw (2,0) circle(2pt);
\end{tikzpicture}
&&
\begin{tikzpicture}
\draw (0,0)--(1,0)--(1.3,.75)--(.5,1.5)--(-.3,.75)--(0,0);
\draw (-1.3,0.75)--(-1,0)--(0,0);\draw (2.3,0.75)--(2,0)--(1,0);
\filldraw (0,0) circle(2pt);\filldraw (1.3,.75) circle(2pt);\filldraw (.5,1.5) circle(2pt);\filldraw (1,0) circle(2pt);\filldraw (-.3,.75) circle(2pt);
\filldraw (-1.3,0.75) circle(2pt);\filldraw (-1,0) circle(2pt);\filldraw (2.3,0.75) circle(2pt);\filldraw (2,0) circle(2pt);
\end{tikzpicture}
\\
\\
\begin{tikzpicture}
\draw (0,0)--(1,0)--(1.3,.75)--(.5,1.5)--(-.3,.75)--(0,0);
\draw (-1,0)--(0,0);\draw (2,0)--(1,0);
\draw (1.3,.75)--(2.3,.75);\draw (-1,0)--(-1.3,.75);
\filldraw (0,0) circle(2pt);\filldraw (1.3,.75) circle(2pt);\filldraw (.5,1.5) circle(2pt);\filldraw (1,0) circle(2pt);\filldraw (-.3,.75) circle(2pt);
\filldraw (-1.3,0.75) circle(2pt);\filldraw (-1,0) circle(2pt);\filldraw (2.3,0.75) circle(2pt);\filldraw (2,0) circle(2pt);
\end{tikzpicture}
&&
\begin{tikzpicture}
\draw (0,0)--(1,0)--(1.3,.75)--(.5,1.5)--(-.3,.75)--(0,0);
\draw (-1,0)--(0,0);\draw (2,0)--(1,0);
\draw (0.5,1.5)--(1.5,1.5);\draw (-1,0)--(-1.3,.75);
\filldraw (0,0) circle(2pt);\filldraw (1.3,.75) circle(2pt);\filldraw (.5,1.5) circle(2pt);\filldraw (1,0) circle(2pt);\filldraw (-.3,.75) circle(2pt);
\filldraw (-1.3,0.75) circle(2pt);\filldraw (-1,0) circle(2pt);\filldraw (1.5,1.5) circle(2pt);\filldraw (2,0) circle(2pt);
\end{tikzpicture}
&&

\begin{tikzpicture}
\draw (1.3,0.75)--(1.6,1.5);
\draw (0,0)--(1,0)--(1.3,.75)--(.5,1.5)--(-.3,.75)--(0,0);
\draw (-1,0)--(0,0);
\draw (1.3,.75)--(2.3,0.75)--(2,0);
\filldraw (0,0) circle(2pt);\filldraw (1.3,.75) circle(2pt);\filldraw (.5,1.5) circle(2pt);\filldraw (1,0) circle(2pt);\filldraw (-.3,.75) circle(2pt);
\filldraw (1.6,1.5) circle(2pt);\filldraw (-1,0) circle(2pt);\filldraw (2.3,0.75) circle(2pt);\filldraw (2,0) circle(2pt);
\end{tikzpicture}
\\
\\
\begin{tikzpicture}
\draw (1.3,0.75)--(1.6,1.5);
\draw (0,0)--(1,0)--(1.3,.75)--(.5,1.5)--(-.3,.75)--(0,0);
\draw (-1,0)--(0,0);\draw (-1.3,.75)--(-.3,.75);
\draw (1.3,.75)--(2.3,0.75)--(2,0);
\filldraw (0,0) circle(2pt);\filldraw (1.3,.75) circle(2pt);\filldraw (.5,1.5) circle(2pt);\filldraw (1,0) circle(2pt);\filldraw (-.3,.75) circle(2pt);
\filldraw (1.6,1.5) circle(2pt);\filldraw (-1,0) circle(2pt);\filldraw (2.3,0.75) circle(2pt);\filldraw (2,0) circle(2pt);\filldraw (-1.3,.75) circle(2pt);
\end{tikzpicture}
&&
\end{tabular}
\end{center}

\begin{proof}
Let $H$ be the subgraph of $G$ isomorphic to one of the above graphs. One can see that $D_3(H)$ is connected. 
Since $G$ is unicyclic, $H$ is 3-induced, and for any $y\in V(G)\setminus V(H)$, there is $x\in V(H)$ such that $d(x,y)\geqslant3$. By Lemma \ref{lem22}, $D_3(G)$ is connected. 
\end{proof}

\begin{lemma}\label{lem211}
Let $G$ be a unicyclic graph with cycle $x_1x_2\cdots x_5x_1$ such that $\deg(x_1)\geqslant3$ and $\deg(x_3)\geqslant3$, $G$ has an inner node $u$ on  the connected component of $G-\{x_2 , x_3, x_4 , x_5\}$ containing $x_1$, with $d(u,x_1)=3k \text{ or } 3k+2$, then $D_3(G)$ is connected.
\end{lemma}
\begin{proof}
By the assumption, if $d(u,x_1)=3k$, $G$ has 3-induced subgraph $H$ isomorphic to the following graph, where $u_{3k}=u$. 
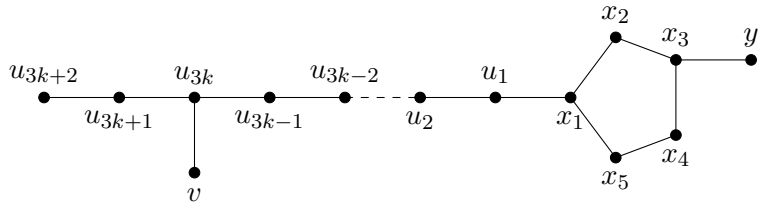
\begin{center}
\begin{tikzpicture}
\draw (1.4,0.5)--(2.4,.5);\draw (-7,0)--(-3,0);\draw [dashed](-3,0)--(-2,0);
\draw (-2,0)--(0,0)--(.6,.8)--(1.4,0.5)--(1.4,-0.5)--(.6,-.8)--(0,0);
\draw (-5,0)--(-5,-1);
\foreach \y in {-7,-6,...,0} \filldraw (\y,0) circle(2pt);
\filldraw (-5,-1) circle(2pt);\filldraw (.6,.8) circle(2pt);\filldraw (1.4,0.5) circle(2pt);\filldraw (1.4,-0.5) circle(2pt);\filldraw (.6,-.8) circle(2pt);
\filldraw (2.4,.5) circle(2pt);
\node at (0,-.3){$x_1$};
\node at (.6,1.1){$x_2$};
\node at (1.4,0.8){$x_3$};
\node at (1.4,-0.8){$x_4$};
\node at (.6,-1.1){$x_5$};
\node at (2.4,0.8){$y$};
\node at (-1,.3){$u_1$};
\node at (-2,-.3){$u_2$};
\node at (-3,.3){$u_{3k-2}$};
\node at (-4,-.3){$u_{3k-1}$};
\node at (-5,.3){$u_{3k}$};
\node at (-6,-.3){$u_{3k+1}$};
\node at (-7,.3){$u_{3k+2}$};
\node at (-5,-1.3){$v$};
\end{tikzpicture}
\end{center}
$u_1x_4$, $u_2x_5yx_1u_3u_6\cdots u_{3k}$, $x_3u_1u_4\cdots u_{3k-2}u_{3k+1}$ and $x_2u_2u_5\cdots u_{3k+2}vu_{3k-2}$ are paths in $D_3(H)$. Thus $D_3(H)$ is connected. By Lemma \ref{lem22}, $D_3(G)$ is connected.

The proof of the case $d(u,x_1)=3k+2$ is similar.
\end{proof}

\begin{lemma}\label{lem212}
Let $G$ be a unicyclic graph with cycle $x_1x_2\cdots x_5x_1$ such that $\deg(x_1)\geqslant3$ and $\deg(x_2)\geqslant3$, $G$ has an inner node $u$ on  the connected component of $G-\{x_2 , x_3, x_4 , x_5\}$ containing $x_1$, with $d(u,x_1)=3k+1\text{ or }3k+2$, then $D_3(G)$ is connected.
\end{lemma}
\begin{proof}
By the assumption, if $d(u,x_1)=3k+1$, $G$ has the following 3-induced subgraph, where $u_{3k+1}=u$. 
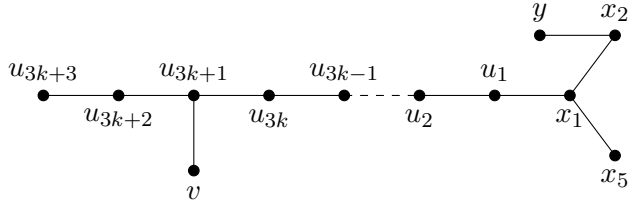
\begin{center}
\begin{tikzpicture}
\draw (0,0)--(.6,.8)--(-.4,.8);\draw (-7,0)--(-3,0);\draw [dashed](-3,0)--(-2,0);
\draw (-2,0)--(0,0)--(.6,-.8);
\draw (-5,0)--(-5,-1);
\foreach \y in {-7,-6,...,0} \filldraw (\y,0) circle(2pt);
\filldraw (-5,-1) circle(2pt);\filldraw (.6,.8) circle(2pt);\filldraw (.6,-.8) circle(2pt);
\filldraw (-.4,.8) circle(2pt);
\node at (0,-.3){$x_1$};
\node at (.6,1.1){$x_2$};
\node at (.6,-1.1){$x_5$};
\node at (-.4,1.1){$y$};
\node at (-1,.3){$u_1$};
\node at (-2,-.3){$u_2$};
\node at (-3,.3){$u_{3k-1}$};
\node at (-4,-.3){$u_{3k}$};
\node at (-5,.3){$u_{3k+1}$};
\node at (-6,-.3){$u_{3k+2}$};
\node at (-7,.3){$u_{3k+3}$};
\node at (-5,-1.3){$v$};
\end{tikzpicture}
\end{center}

this subgraph is isomorphic to $H_{3k+1}$ and by Lemmas \ref{lem22} and \ref{lem24}, $D_3(G)$ is connected.

The proof of the case $d(u,x_1)=3k+2$ is similar.
\end{proof}

\begin{theorem}\label{the213}
Let $G$ be a unicyclic graph with induced cycle $H=x_1x_2\cdots x_5x_1$.
$D_3(G)$ is disconnected if and only if  one of the following conditions is held:
\vspace{-.6cm}
\begin{itemize}
\item
 $H$ has $3$ vertices with degree atleast $3$ and\\
$ i)$  $\deg(x_1)=3$, $\deg(x_2)=\deg(x_5)=2$, $G-\{x_2, x_3, x_4,x_5\}=T\cup sK_1$ where $T$ is a tree and any inner node of $G$ on $T$ has distance $3k+1$ to $x_1$. \\
$ ii)$   $\deg(x_3)=\deg(x_4)=2$, $G-\{x_2, x_3, x_4,x_5\}=T\cup sK_1$ where $T$ is a tree and any inner node on $T$ has distance $3k$ to $x_1$. 
\item
 $H$ has $2$ vertices with degree atleast $3$ and\\
$ i)$  $\deg(x_1)=3$, $\deg(x_2)=\deg(x_3)=\deg(x_5)=2$, $G-\{x_2, x_3, x_4,x_5\}=T\cup sK_1$ where $T$ is a tree and any inner node of $G$  on $T$ has distance $3k+1$ to $x_1$. \\
$ ii)$   $\deg(x_2)=\deg(x_3)=\deg(x_4)=2$, $G-\{x_2, x_3, x_4,x_5\}=T\cup sK_1$ where $T$ is a tree and any inner node in $T$ has distance $3k$ to $x_1$. \item
 $H$ has $1$ vertex with degree atleast $3$ and
$\deg(x_2)=\deg(x_3)=\deg(x_4)=\deg(x_5)=2$, $T=G-\{x_4, x_5\}$  is a tree and the distance of any two inner nodes on $T$ is a multiple of $3$. 
\item
 $G=C_5$.
\end{itemize}
\end{theorem}

\begin{proof}If $G=H$, then $D_3(G)$ is disconnected. By Lemma \ref{lem210}, we can assume that $\deg(x_1)\geqslant3$,
 $H$ has atmost $3$ vertices with degree atleast $3$ in $G$ and $G-\{x_2, x_3, x_4,x_5\}=T\cup sK_1$, where $T$ is a tree containing $x_1$.
\\
If $G-H$ be a null graph, then $D_3(G)$ is disconnected and so we can consider $T-x_1$ has atleast one edge.
\\
 Thus, We have the following cases. \\
Case 1: $H$ has $2$ or $3$ vertices with degree atleast $3$. By Lemma \ref{lem210}, 
$\deg(x_2)=\deg(x_5)=2$ or $\deg(x_3)=\deg(x_4)=2$.
\\
If $\deg(x_2)=\deg(x_5)=2$ and $\deg(x_3)\geqslant3$, then by Lemma \ref{lem211}, $\deg(x_1)=3$ and for any inner node $u$ of $G$ on $T$, $d(u,x_1)=3k+1$. 
By this,  the distance of any two inner nodes of $G$ on $T$ is a multiple of $3$.
Let $a\in N(x_1)\setminus V(H)$. Then the distance of any inner nodes of $T$ to $a$ is a multiple of $3$. Thus $D_3(T)$ is disconnected and $K=\langle\{v : d(v,a)=3k, k=0,1,2,\cdots\}\rangle$ is a connected componet of $D_3(T)$.
Since for any $u\in V(K)\cup\{x_3,x_4\}$ and $v\in V(G-K)\setminus \{x_3,x_4\}$, $d(u,v)\ne3$, $\langle V(K)\cup\{x_3,x_4\}\rangle$ is a connected component of $D_3(G)$ and then $D_3(G)$ is disconnected.
\\
If $\deg(x_3)=\deg(x_4)=2$ and $\deg(x_2)\geqslant3$, then by Lemma \ref{lem212}, for any inner node $u$ of $G$ on $T$, $d(u,x_1)=3k$. 
By Theorem \ref{the27}, $D_3(T)$ is disconnected and $M=\langle\{u\in V(T):d(u,x_1)=3k\}\rangle$ is a connected component of $D_3(T)$. Since for any $u\in V(M)$ and $v\in \{x_2,x_3,x_4,x_5\}\cup N(x_2)\cup N(x_5)$, $d(u,v)\ne3$, $M$ is a connected component of $D_3(G)$ and then $D_3(G)$ is discinnected.
\\
Case 2 :
$x_1$ is the unique vertex of $H$ such that has degree atleast $3$ in $G$. Since $T$ is $3$-induced subgraph of $G$, the connctivity of $D_3(T)$ concludes the connctivity of $D_3(G)$. 
Therefore, $D_3(T)$ is disconnected and so the distance of any two inner nodes in $T$ is a multiple of $3$. By symmetry, in $D_3(G)$ we have $N(x_2)=N(x_5)$ and $N(x_3)=N(x_4)$ and, then  $D_3(G)$ is disconnected too.\end{proof}

Now, we consider unicyclic graphs with $C_4$.

\begin{lemma}\label{lem214}
Let $G$ be a unicyclic graph with cycle $x_1x_2x_3x_4x_1$ such that $\deg(x_1),\deg(x_2)\geqslant3$ and the connected component of $G-\{x_2 , x_3, x_4 \}$ containing $x_1$ has an inner node $u$, with $d(u,x_1)=3k+1 \text{ or } 3k+2$, then $D_3(G)$ is connected.
\end{lemma}
\begin{proof}
Simillar to the proof of Lemma \ref{lem212}, $G$ has a subgraph isopmorphic to $H_{3k+1}$ or $H_{3k+2}$ and then $D_3(G)$ is connected.
\end{proof}

\begin{theorem}\label{the215}
Let $G$ be a unicyclic graph with induced cycle $H=x_1x_2x_3x_4x_1$.
$D_3(G)$ is connected if and only if  $G$ has a $3$-induced subgraph isomorphic to $H_{3k+1}$ or $H_{3k+2}$ for some $k$ or one of the following graphs.

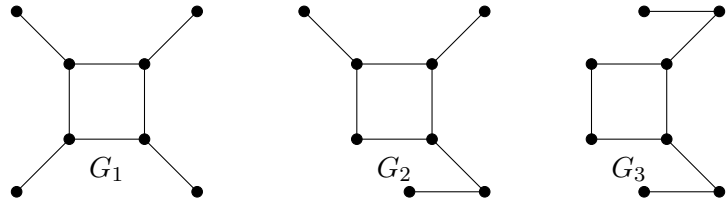
\begin{center}
\begin{tabular}{ccccccc}
\begin{tikzpicture}
\draw (-1.20,-1.2)--(-.5,-.5)--(.5,-.5)--(1.2,-1.2);
\draw (-1.20,1.2)--(-.5,.5)--(.5,.5)--(1.2,1.2);
\draw (-.5,.5)--(-.5,-.5);\draw (.5,.5)--(.5,-.5);
\filldraw (-.5,.5) circle(2pt);\filldraw (-.5,-.5) circle(2pt);\filldraw (.5,.5) circle(2pt);\filldraw (.5,-.5) circle(2pt);\filldraw (1.20,-1.2) circle(2pt);
\filldraw (-1.20,1.2) circle(2pt);\filldraw (1.20,1.2) circle(2pt);\filldraw (-1.20,-1.2) circle(2pt);
\node at (0,-.9){$G_1$};
\end{tikzpicture}
&&&
\begin{tikzpicture}
\draw (-.5,-.5)--(.5,-.5)--(1.2,-1.2)--(.20,-1.2);
\draw (-1.20,1.2)--(-.5,.5)--(.5,.5)--(1.2,1.2);
\draw (-.5,.5)--(-.5,-.5);\draw (.5,.5)--(.5,-.5);
\filldraw (-.5,.5) circle(2pt);\filldraw (-.5,-.5) circle(2pt);\filldraw (.5,.5) circle(2pt);\filldraw (.5,-.5) circle(2pt);\filldraw (1.20,-1.2) circle(2pt);
\filldraw (-1.20,1.2) circle(2pt);\filldraw (1.20,1.2) circle(2pt);\filldraw (.20,-1.2) circle(2pt);
\node at (0,-.9){$G_2$};
\end{tikzpicture}
&&&
\begin{tikzpicture}
\draw (-.5,-.5)--(.5,-.5)--(1.2,-1.2)--(.20,-1.2);
\draw (-.5,.5)--(.5,.5)--(1.2,1.2)--(.2,1.2);
\draw (-.5,.5)--(-.5,-.5);\draw (.5,.5)--(.5,-.5);
\filldraw (-.5,.5) circle(2pt);\filldraw (-.5,-.5) circle(2pt);\filldraw (.5,.5) circle(2pt);\filldraw (.5,-.5) circle(2pt);\filldraw (1.20,-1.2) circle(2pt);
\filldraw (.20,1.2) circle(2pt);\filldraw (1.20,1.2) circle(2pt);\filldraw (.20,-1.2) circle(2pt);
\node at (0,-.9){$G_3$};
\end{tikzpicture}
\end{tabular}
\end{center}
\end{theorem}
\begin{proof}
By contradiction let us assume that $D_3(G)$ is connected but $G$ has no $3$-induced subgraph isomorphic to the above graphs. Since $G$ has no subgraph isomorphic to $G_1$, there is atleast one vertex in $H$ with degree 2. If $H$ has exactly one vertex of degree 2,  we can assume $\deg(x_3)=2$. Since $G$ has no subgraph isomorphic to $G_2$,  all elements of $N(x_2)\setminus V(H)$ and $N(x_4)\setminus V(H)$ are leaves and then $G-\{x_2,x_3,x_4\}=T\cup sK_1$, where $T$ is a tree containing $x_1$. By Lemma \ref{lem214}, $d(u,x_1)=3k$ for any inner node $u$ of $G$ on $T$. But 
$M=\langle\{u\in V(T):d(u,x_1)=3k\}\rangle$ is a connected component of $D_3(G)$, a contradiction. 
\\
Let $H$ has exactly two vertices with degree 2. We have two cases.
\\
\textbf{Case 1:} $\deg(x_1)=\deg(x_2)=2$. Since $G$ has no subgraph isomorphic to $G_3$, we can assume that all elements of $N(x_4)\setminus V(H)$ are leaves. Since $G-x_1$ is a $3$-induced subgraph of $G$ which is a tree, then  it has no $3$-induced subgraph isomorphic to $H_{3k+1}$ or $H_{3k+2}$. 
Hence, the distance of any inner node of $G-x_1$ to $x_3$ is a multiple of 3 and $M=\langle\{u\in V(T):d(u,x_3)=3k\}\rangle$ is a connected component of $D_3(G)$, a contradiction. 
\\
\textbf{Case 2:}
$\deg(x_2)=\deg(x_4)=2$.
Since $N(x_2)=N(x_4)$, $D_3(G)$ is connected if and only if $D_3(G-x_2)$  is connected.
By Theorem \ref{the27}, $D_3(G-x_2)$ is disconnected, a contradiction. Thus, $G$ has three vertices of degree 2. Let $\deg(x_2)=\deg(x_3)=\deg(x_4)=2$. 
Since $N(x_2)=N(x_4)$, a simillar argument gives a contradiction.
\end{proof}

Let $G$ be a graph and $a\in V(G)$. We define $N_3(a)=\{x\in V(G): d(a,x)=3 \}$.
The following theorem charactrizes all unicyclic graph containing $C_3$.

\begin{theorem}\label{the216}
Let $G$ be a unicyclic graph with induced cycle $H=x_1x_2x_3x_1$.
$D_3(G)$ is connected if and only if  $G$ has a $3$-induced subgraph isomorphic to $H_{3k+1}$ or $H_{3k+2}$ for some $k$ or one of the following graphs, where $t=3k\text{ or }3k+2$.

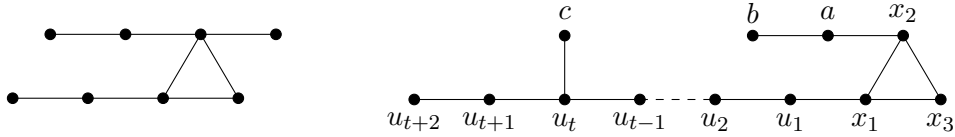
\begin{center}
\begin{tabular}{cccc}
\begin{tikzpicture}
\draw (-2,0)--(1,0)--(.5,.85)--(0,0);
\draw (-1.5,.85)--(1.5,.85);
\filldraw(1,0) circle(2pt);\filldraw[white] (-2,-.5) circle(2pt);
\filldraw (-2,0) circle(2pt);\filldraw (-1,0) circle(2pt);\filldraw (0,0) circle(2pt);\filldraw (-1.5,.85) circle(2pt);\filldraw (-.5,0.85) circle(2pt);\filldraw (.5,0.85) circle(2pt);\filldraw (1.5,0.85) circle(2pt);

\end{tikzpicture}
&&&
\begin{tikzpicture}
\draw (-6,0)--(-3,0);\draw(-2,0)--(1,0)--(.5,.85)--(0,0);
\draw [dashed](-3,0)--(-2,0);
\draw (-1.5,.85)--(.5,.85);\draw (-4,0)--(-4,.85);
\foreach \x in {-6,-5,...,1}\filldraw (\x,0) circle(2pt);
\filldraw (-1.5,.85) circle(2pt);\filldraw (-.5,0.85) circle(2pt);\filldraw (.5,0.85) circle(2pt);\filldraw (-4,0.85) circle(2pt);
\node at (0,-.3){$x_1$};\node at (0.5,1.15){$x_2$};\node at (1,-.3){$x_3$};
\node at (-1,-.3){$u_1$};\node at (-2,-.3){$u_2$};\node at (-3,-.3){$u_{t-1}$};
\node at (-4,-.3){$u_t$};\node at (-5,-.3){$u_{t+1}$};\node at (-6,-.3){$u_{t+2}$};
\node at (-1.5,1.15){$b$};\node at (-0.5,1.15){$a$};\node at (-4,1.15){$c$};
\end{tikzpicture}
\end{tabular}
\end{center}
\end{theorem}

\begin{proof}
By contradiction let us assume that $D_3(G)$ is connected but $G$ has no $3$-induced subgraph isomorphic to the above graphs.
\\
First, let $K= N(x_2)\cup N(x_3)\setminus\{x_1,x_2,x_3\}$, each $y\in K$ be a leaf and $\deg(x_2)\leqslant \deg(x_3)$. Also, let $T$ be the connected component of $G-x_2$ containing $x_1$.
 $D_3(G)$ is connected if and only if $D_3(T)$ is connected because $N_3(x_2)=N_3(x_3)$ and if  $y_2$ and $y_3$ be leaves of $x_2$ and $x_3$, respectively, then $N_3(y_2)\setminus K=N_3(y_3)\setminus K$. 
Thus $D_3(T)$ is connected. Since $T$ is a tree, by Theorem \ref{the27}, $T$ has a 3-induced subgraph isomorphic to $H_{3k+1}$ or $H_{3k+2}$ which is a 3-induced subgraph of $G$, a contradiction.
\\
Thus, we can assume that each of $N(x_1)-\{x_2,x_3\}$ and $N(x_2)-\{x_1,x_3\}$ contain atleast one vertex which is not a leaf. 
By assumption $\deg(x_1)=\deg(x_2)=3$ .
\\
If each $y\in N(x_3)\setminus\{x_1,x_2\}$ be a leaf, then $G-x_3=T\cup sK_1$, where $T$ is a tree and any inner node of $T$ has distance $3k+1$ to $H$. 
\\
Thus, for any  two inner nodes $u$ and $v$ of $T$, $d(u,v)=3k^{\prime}$ and so $D_3(T)$ is disconnected and $A=\langle\{x\in V(T):d(x,w)=3k\}\rangle$ is a connected component of $D_3(T)$, where $w\in N(x_1)\setminus H$. 
Let $b$ be a  leaf of $x_3$. Since $d(b,u)=3$, $d(b,v)=3k$ and $d(x_3,v)\ne3$ for  any $v\in V(T)$, $\langle\{x\in V(G):d(x,w)=3k\}\rangle$  is a connected component of $D_3(G)$, a contradiction.

Then  $N(x_1)-\{x_2,x_3\}$, $N(x_2)-\{x_1,x_3\}$ and $N(x_3)-\{x_1,x_2\}$ contain atleast one vertex which is not a leaf. 
By assumption $\deg(x_1)=\deg(x_2)=\deg(x_3)=3$  and any inner node $u$ of $G$ out of $H$, $d(u,H)=3k+1$ for some $k$.

Similar to the previous case,  $A=\langle\{x\in V(G):d(x,H)=3k+1\}\rangle$ is a connected component of $D_3(G)$, a contradiction.
\end{proof}

By Theorems \ref{the28}, \ref{the29}, \ref{the213}, \ref{the215} and \ref{the216}, all graphs with at most $6$ vertices have disconnected $3$-distance graph. Then we have the following. 

\begin{corollary}
If $G$ is a graph and $D_3(G)$ is connected, then $|G|\geqslant7$.
Furthermore, the cycle $C_7$ is the unique graph with $|V(G)|=7$ such that  $D_3(G)$ is connected.  
\end{corollary}

\section{Conclusion}
In this paper, we focused on the connectivity of the 3-distance graph of some connected graphs. For this, we introduced two new definitions, inner node and generalized double star. We provided some conditions for the connectedness of the $3$-distance graph of an arbitrary graph $G$. Also, we characterized all trees and unicyclic graphs with connected $3$-distance graph. The above achievements raise the following question.

\noindent
\textbf{Problem} : Charactrize all graphs with the connected $3$-distance graph.

\noindent
\textbf{\Large Aknowledgments}\\
%\noindent
The authors would like to thank the referees for their helpful comments.

\noindent
\textbf{\Large Declarations}\\
%\noindent
\textbf{Conflict of interest} No potential conflict of interest was reported by the authors.

\end{document}